\documentclass{article}
\usepackage{amsmath,amssymb,amsthm,amsfonts,enumitem}
\usepackage{color}
\usepackage{graphicx,subcaption}
\usepackage{hyperref}
\usepackage{epstopdf}
\usepackage{bm}
\usepackage{mathtools}
\usepackage{tikz-cd}

\usepackage{pst-node}
\usepackage{tikz-cd} 

\theoremstyle{plain} 
\newtheorem{thm}{Theorem}

\newtheorem{lem}[thm]{Lemma}

\newtheorem{cor}{Corollary}

\theoremstyle{definition}
\newtheorem{defn}{Definition}

\theoremstyle{remark}
\newtheorem*{rem}{Remark}

\usepackage{comment}

\renewcommand{\red}[1]{\textcolor{red}{#1}}

\newcommand{\purple}[1]{\textcolor{purple}{#1}}
\newcommand{\R}{\mathbb{R}}

\begin{document}

\title{Diameter, Area, and Mean Curvature}
\author{Gregory R. Chambers and Jared Marx-Kuo}
\maketitle

\begin{abstract}
In this note, we extend diameter bounds of Simon, Topping, and Wu--Zheng \cite{simon1993existence, Top, wu2011relating} to submanifolds with boundary and (potentially non-compact) ambient manifolds with minor curvature restrictions. The bound is dependent on both an integral of mean curvature and the area of the manifold. We apply our diameter bounds to minimal, constant mean curvature, and prescribed mean curvature surfaces arising in min-max constructions.
\end{abstract}



%

\section{Introduction}
In his seminal work, Simon \cite{simon1993existence} demonstrated that the extrinsic diameter of a closed surface, $M^2 \subseteq \mathbb{R}^3$, has boundable diameter in terms of the area and an integral of mean curvature:
\begin{equation}
d_{ext}(M) \leq \frac{2}{\pi} A(M)^{1/2} \left( \int_M H^2 \right)^{1/2}
\end{equation}
where $d_{ext}$ is the extrinsic diameter. Topping \cite{Top} upgraded this to a bound on \textit{intrinsic} diameter for closed submanifolds $M^m \subseteq \R^n$ as
\begin{equation} \label{toppingBound}
d_{int}(M) \leq C(m) \int_M |H|^{m-1}
\end{equation}
This was improved upon by Wu--Zheng \cite{wu2011relating} to non-compact manifolds with bounded sectional curvature, though with restrictions on the area of the submanifold. We also mention work of Gadgil--Seshadri \cite{gadgil2011surfaces}, Menne--Scharrer \cite{menne2017novel}, Miura \cite{miura2020diameter} who also prove related diameter bounds. \newline
\indent In this note, we aim to relax the type of diameter bound with the added benefit of proving diameter estimates for a larger class of submanifolds. This includes submanifolds with arbitrarily large volume, submanifolds and ambient manifolds with boundary, and ambient manifolds without uniform sectional curvature bounds (see Corollary~\ref*{finiteDiamCor}).
\begin{thm}
\label{thm:main}
Suppose that $M$ is an $m$-dimensional Riemannian manifold without boundary, $N$ is an $n$-dimensional Riemannian manifold, and
$M \rightarrow N$ is an isometric immersion.  We make the following assumptions on $N$:
\begin{enumerate}
\item The sectional curvature of $N$ is bounded above, that is, $K_N \leq k_0$.
\item The injectivity radius of $N$ is bounded below, that is $\textrm{Inj Rad}(N) \geq r_0 > 0$.
\end{enumerate}
Then the intrinsic diameter of $M$, denoted by $d_{int}(M)$, satisfies
\begin{align} \label{ourBound}
 d_{int}(M) &\leq c(m)\int_M |H|^{m-1} d \mu + \frac{r_0}{\overline{V}_{0}}\textrm{Vol}(M) \biggl) \\ \nonumber
 & \leq C(r_0,k_0,m) \biggl( \int_M |H|^{m-1} d \mu + \textrm{Vol}(M) \biggl) 
\end{align}
%
where $H$ is the mean curvature vector of $M$ in $N$ and $\overline{V}_{0} = \overline{V}_0(m, r_0, k_0) > 0$ is a volume constraint depending on the geometry of $N$.
\end{thm}
%
%
\begin{rem} 
We note the scale invariance of equation \eqref{ourBound} as $g \to \lambda g$ implies that $r_0 \to \lambda r_0$ and $\overline{V}_{0} \mapsto \lambda^m \overline{V}_{0}$ (see the proof of \ref{lem:dichotomy} and \S \ref{thm:main}). We also note that $N$ is allowed to have boundary. 
\end{rem}
\noindent We prove a similar bound for submanifolds, $M$, with $\partial M \neq \emptyset$:
%
\begin{thm}
\label{thm:boundary}
Suppose that $M$ is an $m$-dimensional Riemannian manifold with connected boundary, $\partial M \neq \emptyset$, $N$ is an $n$-dimensional Riemannian manifold,
$M \rightarrow N$ is an isometric immersion.  We make the following assumptions on $N$:
\begin{enumerate}
\item The sectional curvature of $N$ is bounded above, that is, $K_N \leq k_0$.
\item The injectivity radius of $N$ is bounded below, that is $\textrm{Inj Rad}(N) \geq r_0 > 0$.
\end{enumerate}
Then the intrinsic diameter of $M$, denoted by $d_{int}(M)$, satisfies
\begin{align} \label{boundaryBound}
 d_{int}(M) & \leq c(m) \Big(\int_M |H|^{m-1} + \frac{r_0}{\overline{V}_{0}}\textrm{Vol}(M) + \textrm{Vol}(M)^{\frac{1}{m}}\Big) + d_{int}(\partial M)\\ \nonumber
 &\leq C(k_0,r_0,m) \Big( \int_M |H|^{m-1} + \max\big(\textrm{Vol}(M), \textrm{Vol}(M)^{\frac{1}{m}} \big) \Big) + d_{int}(\partial M)
\end{align}
where $\overline{V}_{0}$ is the same as in theorem \ref{thm:main}.
\end{thm}
%
%
\begin{rem}
We remark that for Theorems~\ref*{thm:main} and~\ref*{thm:boundary}, we can replace $r_0$
with $\overline{R}(M) = \inf_{x \in M} \textrm{Inj Rad}(N)(x)$ to achieve a potentially stronger bound. In other words, we only need the injectivity radius of $N$ at points in $M$ to be bounded. 
\end{rem}
We refer the reader to work of Wu \cite{wu2023diameter} and Paeng \cite{paeng2014diameter} for related diameter bounds. We emphasize that the diameter bounds from previous literature all forego an area term at the cost of imposing constraints on $\textrm{Vol}(M)$ in terms of the sectional curvature restrictions on the ambient manifold, which are non-trivial unless $K_N$ is non-positive. In some sense, equations \ref{ourBound} and \ref{boundaryBound} provide interpolations of Topping's bound \eqref{toppingBound}, as well as the diameter bound for minimal surfaces coming from the monotonicity formula. Note that
\begin{itemize}
\item For minimal surfaces, equation \eqref{ourBound} gives a stronger bound than the monotonicity formula, which only bounds the \textit{extrinsic} diameter, as opposed to intrinsic.
\item We emphasize that in the statement of Theorem~\ref{thm:main}, there is no restriction on the compactness/non-compactness of $M$ or $N$, and in particular, \textbf{when $M$ is non-compact, Theorem ~\ref{thm:main} shows that either $\int_{M} |H|^{m-1}$ or $\text{Vol}(M)$ is infinite}. This sharpness is exemplified by the plane and Gabriel's horn in $\R^3$, for which bounds of the form $d \leq C \cdot \int_M |H|^{m-1}$ and $d \leq C \cdot A(M)$ fail individually.
\item We pose the following question: are there better bounds in equation \eqref{boundaryBound} using only one of $\textrm{Vol}(M)$ and $\textrm{Vol}(M)^{1/m}$? For example, if $M$ is a minimal surface ($H = 0$) is contained in a small ball in $N$, the monotonicity formula holds for all of $M$, giving a bound in terms of 
$\textrm{Vol}(M)^{1/m}$.
\end{itemize}
We remark that there are modified forms of the monotonicity formula for varifolds with $L^{\infty}$ or $L^p$-bounded mean curvature as in \cite[Thm 17.6, 17.7]{simon1984lectures}. However, these formula do not produce diameter bounds directly. \newline 
\indent We present a few applications of Theorems \ref{thm:main} and \ref{thm:boundary}. Our first application is to submanifolds with bounded mean curvature and follows directly from \ref{thm:main}:
%
\begin{cor} \label{finiteDiamCor} Suppose that $N^{n+1}$ is a manifold with bounded sectional curvature, $M^m \subseteq N^{n+1}$ is a submanifold with prescribed mean curvature (PMC) $H\Big|_M = h$, $H \in L^{m-1}(M)$, $\textrm{Vol}(M) < \infty$, and $\partial M$ is connected with finite diameter (or is empty). Then $M^n$ has finite diameter.
\end{cor}
%
Such submanifolds show up in min-max construction of constant mean curvature (CMC) and prescribed mean curvature (PMC) surfaces (e.g. when the prescribing function is bounded), such as in \cite{chambers2020existence, sun2020multiplicity, stryker2024min, stryker2023localization, mazurowski2022prescribed}. A priori, these surfaces may have infinite diameter, but Corollary~\ref{finiteDiamCor} resolves this. We can also get diameter bounds for non-smooth metrics:
%
%
\begin{cor} \label{finiteDiamCOne} Suppose $M^m$ is a closed submanifold of  $(N^{n+1}, g)$. Further suppose that $g$ is a $C^1$ metric such that $g \Big|_{M \backslash K}$ is $C^2$ for some compact subset $K \subseteq M$. Suppose $M^m \subseteq N^{n+1}$ has bounded mean curvature and finite area, then it has finite diameter. 
\end{cor}
Corollary \ref{finiteDiamCOne} is proved in \S \ref{Applications} and essentially boils down to approximating $g$ with a $C^2$ metric, $g^*$, on an open neighborhood of $K$. Note that this bound is not quantitative (since the choice of smoothing is not unique), but it allows for a diameter bound in terms of the $g$ metric, despite $g$ not having a well defined/bounded sectional curvature tensor everywhere, which is essential in the proof of Theorems \ref{thm:main} and \ref{thm:boundary}. In particular, the proof of Corollary \ref{finiteDiamCOne} indicates that there may be another proof of bounding diameter by mean curvature and volume, using lower regularity techniques. 
%

%
\section{Setup and Lemmas}
For the remainder of the article, we will use the following notation.
$B(x,r)$ denotes the intrinsic open ball in $M$ centered at $x$ of intrinsic radius $r$, and $V(x,r)$ is its volume. 

The proof follows that of Topping.  There are several ingredients, starting with a generalization of the Michael-Simon Sobolev inequality due to Hoffman and Spruck:
\begin{thm}[Thm 2.1 \cite{HS}]
\label{thm:HS}
Let $M \rightarrow N$ be an isometric immersion of Riemannian manifolds of dimension $m$ and $n$, respectively. Assume that $K_N \leq b^2$ for $b \in \mathbb{C}$, and let $h$ be a non-negative $C^1(M)$ (or $W^{1,1}(M)$) function on $M$ vanishing on $\partial M$. Then
$$ \big( \int_M h^{m/(m-1)} d\mu \big)^{(m-1)/m)} \leq c(m) \int_M ( |\nabla h| + h |H| ) d \mu,$$
provided
$$ b^2 (1 - \beta)^{-2/m}(\omega_m^{-1} Vol(\textrm{supp} \; h))^{2/m} \leq 1$$
and
$$ 2 \rho_0 \leq \overline{R}(M),$$
where
$$ \rho_0 = b^{-1} \sin^{-1} \big( b(1-\beta)^{-1/m} (\omega_m^{-1} Vol(\textrm{supp} \; h))^{1/m} \big) $$
if $b$ is real, and if $b$ is imaginary, then
$$ \rho_0 = (1-\beta)^{-1/m} (\omega_m^{-1} Vol(\textrm{supp} \; 
h)^{1/m}). $$
Here, $\overline{R}(M)$ is the injectivity radius of $N$ restricted to $M$.  In addition, $\beta$ is a free parameter,
$0 < \beta < 1$, and
$$ c(m) = c(m,\beta) = \pi \cdot 2^{m-1} \beta^{-1} (1 - \beta)^{-1/m} \frac{m}{m-1} \omega_m^{-1/m}.$$
\end{thm}
\noindent We note that for a fixed value of $\beta$ and function $h$, the conditions of theorem \ref{thm:HS} are scale invariant as $g \mapsto \lambda g$ gives $K \mapsto \lambda^{-1} K$, $\text{Vol}(\textrm{supp} \; h) \to \lambda^m \text{Vol}(\textrm{supp} \; h)$, and $\overline{R}(M) \to \lambda \overline{R}(M)$. When $b \in \mathbb{R}$, one can verify that each condition scales appropriately, and when $b \in i\mathbb{R}$, $\overline{R}(M) = \infty$ by Cartan-Hadamard, and both conditions automatically hold. \newline 
%
%
\indent For $b \in \R$, we define $V_0 = V_0(b, r_0, \beta) > 0$ to be the smallest of the two positive solutions to 
\begin{align} \label{VNoughtDef}
b^2 (1 - \beta)^{-2/m}(\omega_m^{-1} V)^{2/m} &= 1 \\ \nonumber
2 \frac{1}{b} \sin^{-1}(b (1 - \beta)^{-1/m}(\omega_m^{-1} V)^{-1/m}) &= r_0
\end{align}
when $\frac{|b|}{2} r_0 \in \left(0, \frac{\pi}{2} \right]$. If $\frac{|b|}{2} r_0 > \frac{\pi}{2}$, we ignore the second equation and define $V_0$ according to the first. Since $\overline{R}(M) \geq r_0$, the defined value of $V_0$ will satisfy both constraints of theorem \ref{thm:HS}. \newline 
\indent We define two quantities, in the same vein as Topping.
\begin{defn}
\label{defn:maximal}
$$ M(x,R) = \sup_{r \in (0,R]} r^{-\frac{1}{m-1}} \big( V(x,r) \big)^{-\frac{m-2}{m-1}} \int_{B(x,r)} |H| d \mu$$
\end{defn}

\begin{defn}
\label{defn:minimal}
$$ \kappa(x,R) = \inf_{r \in (0,R]} \frac{V(x,r)}{r^m}$$
\end{defn}

Using the Hoffman-Spruck version of the Michael-Simon Sobolev inequality (Theorem~\ref*{thm:HS}), we can prove the following
lemma, analogous to \cite[Lemma 1.2]{Top}:

\begin{lem}
\label{lem:dichotomy}
As before, let $M$ and $N$ be $m$- and $n$-dimensional Riemannian manifolds, respectively.  If $M$ is closed and isometrically
immersed in $N$, and if $N$ is complete with uniformly bounded sectional curvature $\sec \leq b^2$, then there exists a constant $\delta = \delta(m) > 0$ and for all $x \in M$, there exists $R_x \in (0, r_0]$ so that at least 
one of the following is true:
\begin{enumerate}
\item $M(x, R_x) > \delta$
\item $\kappa(x, R_x) > \delta$
\end{enumerate}
\end{lem}
\begin{proof}
%
%
As in Topping \cite[Lemma 1.2]{Top}, suppose $\delta > 0$ to be chosen later, and suppose $M(x,r) \leq \delta$ for some $r \in [0,r_0]$. Let $R_x$ be the largest of the values in $[0, r_0]$ for which $M(x, R_x) \leq \delta$ and the hypothesis of theorem \ref{thm:HS} are satisfied when $\textrm{supp} \; h = B(x, R_x)$ (say for $\beta = 1/2$). By definition
\[
\int_{B(x,r)} |H| d \mu \leq \delta r^{1/(m-1)} V(x,r)^{(m-2)/(m-1)}
\]
for all $r \in [0,R_x]$. 
%
%
We consider the following Lipschitz function on $M$
\[
f(y) = \begin{cases}
1 & y \in B(x,r) \\
1 - \frac{1}{\mu}(\text{dist}_M(x,y) - r) & y \in B(x,r+\mu) \backslash B(x,r) \\
0 & y \not \in B(x,r+\mu)
\end{cases}
\]
\noindent We choose $r < R_x$ and consider all $\mu$ sufficiently small so that the hypothesis of Theorem \ref{thm:HS} are satisfied for $\text{supp}(f) = B(x,r+\mu)$. Applying \ref{thm:HS} (for again, $\beta = 1/2$), we compute
\[
Vol(x,r)^{(m-1)/m} \leq c(m) \left[ \frac{1}{\mu}\left(V(x,r+\mu) - V(x,r) \right) + ||H||_{L^1(B(x,r))} \right]
\]
sending $\mu \to 0$, we conclude 
\[
V(x,r)^{(m-1)/m} \leq c(m) \left[ \frac{dV(x,r)}{dr} + ||H||_{L^1(B(x,r))} \right] 
\]
using our assumed bound on the mean curvature, we get 
\begin{equation} \label{ODEInequality}
\frac{dV(x,r)}{dr} + \delta r^{1/(m-1)} V(x,r)^{(m-2)/(m-1)} - c(m)^{-1} V(x,r)^{(m-1)/m} \geq 0
\end{equation}
we now proceed exactly as in Topping \cite[Lemma 1.2]{Top}, and see that one can choose $\delta$ sufficiently small so that $v(r) = \delta r^m$ satisfies the reverse inequality \eqref{ODEInequality}. Noting that $\lim_{r \to 0} \frac{V(x,r)}{v(r)} = \frac{\omega_m}{\delta}$, we choose $\delta < \min(\omega_m, \delta_0)$ where $\delta_0$ is the smallest positive value satisfying 
\[
m \delta_0 + m \delta^{\frac{2m-3}{m-1}}_0 - c(m)^{-1} \delta^{\frac{m-1}{m}}_0 = 0
\]
With this choice of $\delta = \delta(m)$ and the ODE inequalities for $V(x,r)$ and $v(r)$, we conclude that  $V(x,r) > v(r)$ 
so that 
\[
\kappa(x,R_x) = \inf_{r \in (0,R_x]} \frac{V(x,r)}{r^m} > \delta
\]
%
%
\end{proof}
\noindent Now, for every point $x \in M$, one of the two properties from Lemma~\ref*{lem:dichotomy} holds. 
%
If the second property is true, then we can choose a ball centered at $x$ with the following properties:
\begin{enumerate}
\item Its radius is $r_x \leq r_0$, and
\item $V(x,r_x) \geq \delta r_x^m $
\end{enumerate}
We now ``post-process" the balls in the second case:
\begin{lem} \label{postProcess}
Suppose $x \in M$ and $B(x, r_x)$ only satisfies the second case of Lemma \ref{lem:dichotomy}. Then there exists a radius $\tilde{r}_x \in [r_x, r_0]$ such that at least one of the following holds
\begin{enumerate}
    \item $$ \tilde{r}_x^{-\frac{1}{m-1}} \big( V(x,\tilde{r}_x) \big)^{-\frac{m-2}{m-1}} \int_{B(x,\tilde{r}_x)} |H| d \mu > \delta $$
\item $\textrm{Vol}(x, \tilde{r}_x) \geq \min(\delta r_0^m, V_0) = V_{0, \delta}$
\end{enumerate}
\end{lem}
\begin{proof}
Suppose the first case in the lemma does not hold. Because $B(x,r_x)$ only satisfies the second case of Lemma \ref{lem:dichotomy}, then we conclude
\[
M(x, r_0) = \sup_{r \in (0,r_0]} r^{-\frac{1}{m-1}} \big( V(x,r) \big)^{-\frac{m-2}{m-1}} \int_{B(x,r)} |H| d \mu \leq \delta 
\]
Now let $R_x$ be the radius such that $Vol(x, R_x) = V_0$ as in equation \eqref{VNoughtDef}. If $R_x > r_0$, then the same ODE argument as in Lemma \ref{lem:dichotomy} allows us to conclude that $V(x, r) > \delta r^m$ for all $r \in [0, r_0]$, and so 
\[
V(x, r_0) > \delta r_0^m
\]
setting $\tilde{r}_x = r_0$ gives the second case of our lemma. If $R_x \leq r_0$, we set $\tilde{r}_x = R_x$ and conclude $V(x, \tilde{r}_x) = V_0$.
\end{proof}
\begin{rem}
Lemma \ref{postProcess} tells us that we can always upgrade the balls provided by Lemma \ref{lem:dichotomy} to potentially larger balls so that either the mean curvature bound holds, or the ball has ``large" volume. 
\end{rem}
\section{Proof of Theorem \ref{thm:main}}
We can now prove Theorem~\ref*{thm:main}.  Fix $a$ and $b$ in $M$, and choose a length minimizing geodesic $\gamma$ in $M$ which connects $a$ and $b$. Cover $\gamma$ with balls coming from Lemma \ref{lem:dichotomy}, which are then potentially replaced ($r_x \mapsto r_x = \tilde{r}_x$) according to Lemma \ref{postProcess}. This collection of balls cover $\gamma$. Since $\gamma$ is compact, there is a finite subset which covers $\gamma$.  Since the curve is a geodesic, each ball intersects it in an open interval.
This is because we are assuming that $r_x \leq r_0$, and so the shortest point from $x$ to a point on the boundary of
the geodesic ball of radius $r_x$ centered at $x$ entirely lies in
that ball.

Furthermore, we note that if $x$ and $y$ are points on the curve,
and if $B_x$ and $B_y$ are geodesic balls centered at $x$ and $y$ (respectively), then we cannot have both
$$ \overline{B_x} \cap \overline{B_y} \neq \emptyset$$ and $$ \overline{B_x} \cap \overline{B_y} \cap \gamma = \emptyset. $$
By the finite version of the Vitali Covering Lemma, there is a disjoint subset of these balls so that their triples cover $\gamma$.  Thus, the disjoint balls cover at least $1/10$ of the length of $\gamma$, call this collection $\{B(x_i, r_i)\}$. \newline 
\indent We now divide up the balls into two categories; ones that come from the first property and ones that come from the second property of lemma \ref{postProcess}. For the ones
that come from the first property, by the same argument as in \cite[Thm 1]{Top} (choose $\lambda = 1/10$ in the proof), the sum of their radii is bounded above by 
$20 \delta^{1-m} \int_M |H|^{m-1} d \mu = c(m) \int_M |H|^{m-1} d \mu$, where we emphasize that the constant is independent of the sectional curvature bound. \newline 
\indent For the second collection, each ball is disjoint and has volume at least $V_{0, \delta}$.  Since they are all disjoint, there can only be at most $$ \frac{\textrm{Vol}(M)}{V_{0, \delta}}, $$ and each has radius $r_x \leq r_0$.  Thus, the total length of $\gamma$ covered by these
balls is bounded above by
$$ r_0 \frac{\textrm{Vol}(M)}{V_{0, \delta}}.$$
\noindent Since $\delta = \delta(m)$ and $V_{0,\delta} = C(m, k_0, r_0)$ just depend on the geometry of $N$, this completes the proof.

\section{Proof of Theorem \ref{thm:boundary}}
To start the proof, we prove an analogous version of \ref{lem:dichotomy}, but in the setting of non-empty boundary.
\begin{lem}
\label{lem:dichotomyBoundary}
As before, let $M$ and $N$ be $m$- and $n$-dimensional Riemannian manifolds, respectively and $\partial M \neq \emptyset$.
If $M$ is isometrically immersed in $N$, and if $N$ is complete with uniformly bounded sectional curvature $\sec \leq b^2$, then there exists a constant $\delta = \delta(m) > 0$ so that for all $x \in M$, there exists a $R_x \in (0, \min(r_0,\textrm{dist}(x,\partial M))]$ and at least one of the following is true:
\begin{enumerate}
\item $M(x,R_x) > \delta$
\item $\kappa(x,R_x) > \delta$
\end{enumerate}
Here $\textrm{dist}(x,\partial M)$ is the distance from $x$ to the boundary of $M$.
\end{lem}
\begin{proof}
The proof of Lemma \ref{lem:dichotomyBoundary} is identical to \ref{lem:dichotomy} as long as we take our ball radii smaller than $\text{dist}(x, \partial M)$ in order to avoid the boundary.
\end{proof}  
\noindent We similarly prove a ``post-processing" lemma for the boundary case
\begin{lem} \label{postProcessBoundary}
Suppose $x \in M$, $\textrm{dist}(x, \partial M) \geq r_0$, and $B(x, r_x)$ only satisfies the second case of Lemma \ref{lem:dichotomy}. Then there exists a radius $\tilde{r}_x \in [r_x, r_0]$ such that at least one of the following holds
\begin{enumerate}
    \item $$ \tilde{r}_x^{-\frac{1}{m-1}} \big( V(x,\tilde{r}_x) \big)^{-\frac{m-2}{m-1}} \int_{B(x,\tilde{r}_x)} |H| d \mu > \delta $$
\item $\textrm{Vol}(x, \tilde{r}_x) \geq \min(\delta r_0^m, V_0) = V_{0, \delta}$
\end{enumerate}
\end{lem}
\begin{proof}
The proof is exactly the same as lemma \ref{postProcess} since we assume $\textrm{dist}(x, \partial M) \geq r_0$.
\end{proof}
\noindent We now prove Theorem \ref{thm:boundary}. WLOG, we assume that $\text{diam}(\partial M) < \infty$, and hence $\partial M$ is compact. Fix $x \in M$, and choose $y \in \partial M$ so that:
$$ \textrm{dist}(x,\partial M) = \textrm{dist}(x,y).$$  Such
a $y$ exists since $\partial M$ is closed, and $M$ is connected. We show the following
%
%
\begin{lem} \label{distToBoundary}
Let $\gamma$ be a distance minimizing curve from $x \in \mathring{M}$ to $\partial M$. Then 
\[
\ell(\gamma) \leq c(m) \Big(\int_M |H|^{m-1} + \frac{r_0}{V_{0,\delta}} \textrm{Vol}(M) + \textrm{Vol}(M)^{1/m}\Big) 
\]    
\end{lem}
\noindent We note that Lemma \ref{distToBoundary} resolves theorem \ref{thm:boundary}, for if $x$ and $z$ are points in $M$, then we can form a path from one to the other
by choosing a length minimizing curve, $\alpha$, from $x$ to $\partial M$, choosing a length minimizing curve, $\beta$, from
$z$ to $\partial M$, and then joining the two points on $\partial M$ by a length minimizing curve. This uses that $\partial M$ is connected (though lemma \ref{distToBoundary} does not require connectedness) This yields the bound:
$$ \textrm{dist}(x,z) \leq \textrm{length}(\alpha) + \textrm{length}(\beta) + \textrm{diam}(\partial M) $$
which yields theorem \ref{thm:boundary}.

We return to the curve $\gamma$ which connects $x$ to $\partial M$ and is length minimizing.  To bound the length of $\gamma$, 
consider
\[
M_{\epsilon} = \{x \in M \; | \; \text{dist}(x, \partial M) \geq \epsilon \}
\]
which is well defined for all $\epsilon$ sufficiently small. We cover $M_{\epsilon} \cap \gamma$ by 
\[
M_{\epsilon} \cap \gamma \subseteq  \bigcup_{x \in M_{\epsilon} \cap \gamma} B(x,r_x)
\]
where the collection of balls $\{B(x, r_x)\}$ are chosen according to lemma \ref{lem:dichotomyBoundary}, refining those with $\textrm{dist}(x, \partial M) \geq r_0$ with lemma \ref{postProcessBoundary}. Moreover, for each such $x$, choose $r_x$ to be the \textit{largest} such ball given by lemmas \ref{lem:dichotomyBoundary} \ref{postProcessBoundary}. As in the proof of theorem \ref{thm:main}, each of these balls intersects $\gamma$ in an interval, and we cannot have both
$$ \overline{B_x} \cap \overline{B_y} \neq \emptyset$$ and $$ \overline{B_x} \cap \overline{B_y} \cap \gamma = \emptyset. $$  Since these balls
cover $\gamma$, we use the 
Vitali Covering Lemma to find a finite collection of disjoint balls, $\{B(x_i, r_i)\}$, which cover at least $\ell(\gamma)/10$. \newline 
\indent Consider all balls for which $\text{Vol}(B(x_{i}, r_{i})) > V_{0,\delta}$ \textbf{or} $M(x_{i}, r_{i}) > \delta$, and call this subcollection $\{B(x_{i_j}, r_{i_j})\}$. Note that this subcollection contains all balls $\{B(x_{i}, r_{i})\}$ which have $\text{dist}(x_{i}, \partial M) \geq r_0$ as lemmas \ref{lem:dichotomyBoundary} \ref{postProcessBoundary} apply.   We now argue exactly as in the proof of theorem \ref{thm:main} and conclude
\[
\sum_{j = 1}^N r_{i_j} \leq c(m) \left[ \int_{M} |H|^{m-1} + \frac{r_0}{V_{0,\delta}} \textrm{Vol}(M)\right]
\]
%
We now handle the remaining balls, which necessarily satisfy $\text{dist}(x_i, \partial M) < r_0$, $M(x_i, r_i) \leq \delta$, and $\text{Vol}(B(x_i, r_i)) \leq V_{0,\delta}$. We will show that there is at most one ball satisfying these conditions. Consider, 
$x_i$ satisfying our conditions \textit{and} such that $\text{dist}(x_i, \partial M)$ is maximized. Further consider the ball $B(x_i, \overline{r}_i)$ with $\overline{r}_i = \text{dist}(x_i, \partial M)$. Note that if $r_i < \overline{r}_i$, then necessarily $M(x_i, \overline{r}_i) \leq \delta$, else $r_i$ wouldn't have been the \textit{largest} radii satisfying one of the conclusions of lemma \ref{lem:dichotomyBoundary} (which we enforced during construction). See figure \ref{fig:oneball} for a visualization.
\begin{figure}[h!]
    \centering
    \includegraphics[scale=0.3]{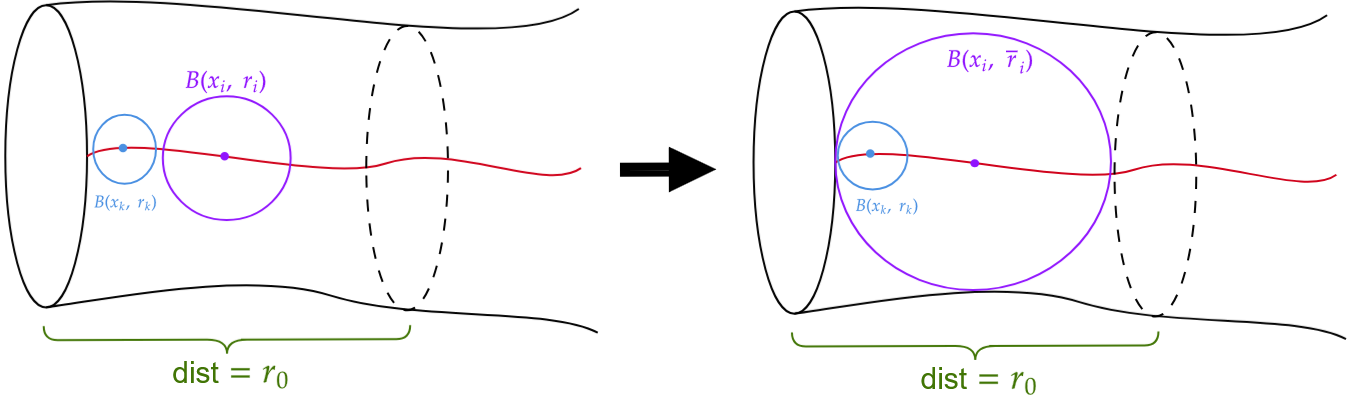}
    \caption{Visualization of one ball argument}
    \label{fig:oneball}
\end{figure}
\newline Now the same ODE argument as in the proof of lemma \ref{lem:dichotomyBoundary} tells us that $\kappa(x_i, \overline{r}_i) > \delta$, and so again, $r_i$ wouldn't have been the maximal choice of radius. Thus $r_i = \overline{r}_i = \text{dist}(x_i, \partial M)$. But now, disjointness of our finite subcover tells us that no other $x_k$ satisfies the three conditions, $\text{dist}(x_k, \partial M) < r_0$, $M(x_k, r_k) \leq \delta$ and $\text{Vol}(B(x_k, r_k)) \leq V_{0,\delta}$, else $B(x_k, r_k) \cap B(x_i, r_i) \neq \emptyset$. \newline 
\indent Thus it suffices to bound $r_i$. Note that $\kappa(x_i, r_i) \geq \delta$, so 
\[
V(x_i, r_i) \geq \delta r_i^m \implies r_i \leq \frac{V(x_i, r_i)^{1/m}}{\delta^{1/m}} \leq \frac{\text{Vol}(M)^{1/m}}{\delta^{1/m}}
\]
%
%
From this, we now conclude
\begin{align*}
\ell(M_{\epsilon} \cap \gamma) &\leq c(m) \left[ \int_{M} |H|^{m-1} + \frac{r_0}{V_{0,\delta}} \textrm{Vol}(M)\right] + \frac{\textrm{Vol}(M)^{1/m}}{\delta^{1/m}} \\
& \leq c(m) \Big( \int_M |H|^{m-1} + \frac{r_0}{V_{0,\delta}} \textrm{Vol}(M) + \textrm{Vol}(M)^{1/m} \Big)
\end{align*}
Sending $\epsilon \to 0$, we conclude the bound on $\ell(\gamma)$. 
%
This concludes the proof of lemma \ref{distToBoundary} and setting $C(r_0, k_0, m) = c(m) \frac{r_0}{V_{0,\delta}}$, the proof of theorem \ref{thm:boundary}. 
\section{Applications} \label{Applications}
As an application of Theorem \ref{thm:main} we prove Corollary \ref{finiteDiamCor}:
%
\begin{cor} 
Suppose that $N^{n+1}$ is a manifold with bounded sectional curvature, $M^m \subseteq N^{n+1}$ is a submanifold with prescribed mean curvature (PMC) $H\Big|_M = h$, $H \in L^{m-1}(M)$, $\textrm{Vol}(M) < \infty$, and $\partial M$ is connected with finite diameter (or is empty). Then $M^n$ has finite diameter.
\end{cor}
\begin{proof}
Consider a sequence of pairs of points $\{(x_i, y_i)\}$ along with distance minimizing geodesics $\gamma_i: x_i \to y_i$ such that 
\[
d_{int}(Y) = \lim_{i \to \infty} \text{dist}(x_i, y_i) = \lim_{i \to \infty} \ell(\gamma_i)
\]
Then recreating the proof of Theorem \ref{thm:main} or Theorem \ref{thm:boundary} shows that 
\[
\ell(\gamma_i) \leq \begin{cases}
C(k_0, r_0, m) \Big(\int_M |H|^{m-1} + \text{Vol}(M)\Big) & \partial M = \emptyset \\
C(k_0, r_0, m) \Big(\int_M |H|^{m-1} + \text{Vol}(M)\Big) + d_{int}(\partial M) & \partial M \neq \emptyset
\end{cases}
\]
this bound is uniform in $i$ so we conclude. 
\end{proof}
\noindent We now prove Corollary \ref{finiteDiamCOne}
\begin{cor} 
%
Suppose $(N^{n+1}, g)$ is a Riemannian $(n+1)$-manifold with $g$ a $C^1$ metric such that $g \Big|_M$ is $C^2$ outside some compact subset $K \subseteq M$. Suppose $M^m \subseteq N^{n+1}$ has bounded mean curvature and finite area, then it has finite diameter. 
\end{cor}
\noindent \textbf{Proof}: Consider a (precompact) open neighborhood $U$ with $K \subseteq U \subseteq N$ and mollify $g$ against a smooth bump function on $U$. Let the resulting metric be $g^*$, and note that $g^*$ is a $C^2$ metric on $N$, agrees with $g$ outside a compact set in $N$, and $||g - g^*||_{C^1(N)} = C_0 < \infty$. 
Then we see that
%
\begin{align*}
\text{d}_{int, g}(M) &\leq (1 + ||g - g^*||_{C^0}) d_{int, g^*}(M) \\
\textrm{Vol}_{g^*}(M) & \leq (1 + ||g - g^*||_{C^0}) \textrm{Vol}_{g}(M)
\end{align*}
Moreover,
\[
||H_{M, g} - H_{M, g^*}||_K \leq C(m, ||g||_{C^1(U)}, ||g^*||_{C^1(U)}) ||g - g^*||_{C^1(U)}
\]
In particular, if $H_{M, g}$ is bounded in $L^{\infty}(M)$, then we see that $H_{M, g^*}$ is as well. We can now compute:
\begin{align*}
d_{int,g}(M) & \leq (1 + ||g - g^*||_{C^0}) d_{int, g^*}(M) \\
& \leq (1 + ||g - g^*||_{C^0}) C(k_0(g^*), r_0(g^*), m) \left( \int_M |H_{M, g^*}|^{m-1} + \text{Vol}_{g^*}(M) \right) \\
& \leq C(k_0(g^*), r_0(g^*), m, ||g||_{C^1(U)}, ||g^*||_{C^1(U)}, ||g - g^*||_{C^1(U)}) \\
& \qquad \qquad \cdot \left( \int_M |H_{M, g}|^{m-1} + \text{Vol}_{g}(M) \right) \\
& < \infty
\end{align*}
having applied theorem \ref{thm:main} with the metric $g^*$. In the case with boundary, we do the analogous proof (assuming that $d_{int, g}(\partial M)$ is finite). \qed 

\bibliographystyle{amsplain}
\bibliography{bibliography}

\end{document}